\documentclass[11pt]{amsart} %

\usepackage{float}

\usepackage{epsfig}
\usepackage{epstopdf}
\usepackage{tikz}
\usepackage{array,amsmath, enumerate,  url, psfrag}
\usepackage{amssymb, amsaddr, fullpage}
\usepackage{graphicx,subfigure}
\usepackage{color}
\usepackage{amssymb}
\theoremstyle{plain}
\newtheorem{thm}{Theorem}[section]

\newtheorem{claim}[thm]{Claim}
\newtheorem{theorem}{Theorem}
\numberwithin{theorem}{section}

\numberwithin{lemma}{section}

\numberwithin{corollary}{section}

\newtheorem{conjecture}{Conjecture}
\numberwithin{conjecture}{section}
\newtheorem*{conjecture*}{Conjecture}

\title{Enhancing the Erd\H{o}s-Lov\'asz Tihany Conjecture for line graphs of multigraphs}

\author{Yue Wang$^{1}$,  Gexin Yu$^{2}$}

\address{
$^{1}$\small School of Mathematics, Shandong University, Jinan, Shandong, China.\\
$^{2}$\small Department of Mathematics, William \& Mary, Williamsburg, VA, USA.
}

\thanks{The work was done while the first author was at William \& Mary as a visiting student, partially  supported by the Chinese Scholarship Council.  The research of the last author was supported in part by a summer research grant from William \& Mary.}

\email{m15064013175@163.com(Y. Wang), gyu@wm.edu}

\begin{document}

\maketitle

\begin{abstract}
In this paper, we prove an enhanced version of the Erd\H{o}s-Lov\'asz Tihany Conjecture for line graphs of multigraphs. That is, for every line graph $G$ whose chromatic number $\chi(G)$ is more than its clique number $\omega(G)$ and for any nonnegative integer $\ell$, any two integers $s,t \geq 3.5\ell+2$ with $s+t = \chi(G)+1$, there is a partition $(S,T)$ of the vertex set $V(G)$ such that $\chi(G[S])\geq s$ and $\chi(G[T])\geq t+\ell$. In particular, when $\ell=1$, we can obtain the same result just for any $s,t\geq4$. The Erd\H{o}s-Lov\'asz Tihany conjecture for line graphs is a special case when $\ell=0$.
\end{abstract}

\section{Introduction}\label{sect1}

For a (multi)graph $G=(V,E)$, let the {\em clique number} $\omega(G)$ be the size of largest clique in $G$ and the {\em chromatic number} $\chi(G)$ be the smallest integer $k$ such that $V(G)$ can be partitioned into $k$ independent sets.  For a subset $S\subseteq V(G)$, let $G[S]$ be the induced subgraph of $G$ by $S$.  In 1968, Erd\H{o}s and Lov\'asz~\cite{EL68} made the following famous conjecture:

\begin{conjecture}\label{con1}\emph{(Erd\H{o}s-Lov\'asz Tihany Conjecture)}.
For every graph $G$ with $\chi(G) > \omega(G)$ and any two integers $s,t\geq 2$ with $s+t =\chi(G)+1$, there is a partition $(S,T)$ of the vertex set $V(G)$ such that $\chi(G[S])\geq s$ and $\chi(G[T])\geq t$.
\end{conjecture}

The only settled cases of this conjecture are $(s,t)\in\{(2,2),(2,3),(2,4),(3,3),(3,4),(3,5)\},$  see \cite{BJ99, M86, S87a, S87b}. This conjecture is also known to be true for some special classes of graphs, such as line graphs of multigraphs (Kostochka and Stiebitz \cite{KS08}), quasi-line graphs, graphs with independence number two (Balogh, Kostochka, Prince and Stiebitz~\cite{BKPS09}) and graphs with forbidden holes (Song~\cite{S19}).

More recently, the following two relaxed version of Conjecture \ref{con1} were proved. For every claw-free graph $G$ with $\chi(G)>\omega(G)$, there exists a clique $K$ with $|V(K)|\le5$ such that $\chi(G\backslash V(K))>\chi(G)-|V(K)|$ (Chudnovsky, Fradki and Plumettaz~\cite{CFP}). And for integers $s,t\ge2$, any graph $G$ with $\omega(G)<\chi(G)=s+t-1$ contains disjoint subgraphs $G_1$ and $G_2$ of $G$ with either $\chi(G_1)\ge s$ and $col(G_2)\ge t$, or $col(G_1)\ge s$ and $\chi(G_2)\ge t$, where $col(H)$ denotes the coloring number of a graph $H$ (Stiebitz~\cite{S17}).

A connected graph $G$ is {\em double-critical} if $\chi(G)=t$ but $\chi(G\backslash\{x,y\})=t-2$ for every edge $xy\in E(G)$.  The following well-known conjecture is the case of $s=2$ of Conjecture~\ref{con1}.

\begin{conjecture}\label{con2}\emph{(Double-Critical Graph Conjecture~\cite{EL68})}.
For $t\ge 3$, the only double-critical $t$-chromatic graph is $K_t$.
\end{conjecture}

From \cite{S87a}, Conjecture~\ref{con2} holds when $t\le 5$. For $t\ge 6$, Conjecture \ref{con2} remains wide open, and we even do not know if every double-critical $t$-chromatic graph contains $K_4$ as a subgraph. As Conjecture~\ref{con1} implies Conjecture~\ref{con2}, Conjecture \ref{con2} is true for some special classes of graphs as mentioned above. The following are some results related to Conjecture \ref{con2}. Kawarabayashi, Pedersen and Toft~\cite{KPT10} have shown that any double-critical, $t$-chromatic graph contains $K_t$ as a minor for $t\in\{6,7\}$. Pedersen ~\cite{P11} showed that any double-critical, 8-chromatic graph contains $K_{8^-}$ as a minor. Albar and Gon\c{c}alves~\cite{AG18} later proved that any double-critical, 8-chromatic graph contains $K_8$ as a minor. Their proof is computer-assisted. Rolek and Song~\cite{RS18} gave a computer-free proof of the same result and further showed that any double-critical, $t$-chromatic graph contains $K_9$ as a minor for all $t\ge9$. Recently, Huang and Yu~\cite{HY} proved that the only double-critical, 6-chromatic, claw-free graph is $K_6$. Rolek and Song~\cite{RS17} further proved that the only double-critical, $t$-chromatic, claw-free graph is $K_t$ for all $t\le8$.

In this paper, we would like to study a enhanced version of Conjecture~\ref{con1} that for nonnegative integer $\ell$, if there exists a partition $(S,T)$ of the vertex set $V(G)$ such that $\chi(G[S])\geq s$ and $\chi(G[T])\ge t+\ell$ and how much $\ell$ can be?

As a starting point, we consider line graphs.  It turns out that Conjecture~\ref{con1} can be greatly enhanced: for any nonnegative integer $\ell$ when $s,t\ge 3.5\ell+2$, we can find a clique of size $s$ to delete in the line graph $L(G)$ such that the chromatic number of the remaining graph is at least $t+\ell$ (other than $t$ in Conjecture~\ref{con1}).

\begin{theorem} \label{th1}
Let $s, t$ and $\ell$ be arbitrary integers with $3.5\ell+2\leq s\leq t$, $\ell\geq 0$. If the line graph $L(G)$ of some multigraph $G$ has chromatic number $s+t-1> \omega(L(G))$, then it contains a clique $Q$ of size $s$ such that $\chi(L(G)-Q)\geq t+\ell$.
\end{theorem}

Note that when $\ell=0$, Theorem~\ref{th1} implies the result of Kostochka and Stiebitz \cite{KS08}.  When $\ell=1$, the bounds on $s,t$ can be made a little tighter.

\begin{theorem} \label{th2}
Let $s$ and $t$ be arbitrary integers with $4\leq s\leq t$. If the line graph $L(G)$ of some multigraph $G$ has chromatic number $s+t-1> \omega(L(G))$, then it contains a clique $Q$ of size $s$ such that $\chi(L(G)-Q)\geq t+1$.
\end{theorem}

In the next section, we prove the main results.  In the final section, we have further discussion.

\section{Proof of Theorem \ref{th1} and Theorem \ref{th2}}\label{sect2}

In this section, we prove Theorem \ref{th1} and Theorem \ref{th2}.

The case $\ell=0$ is the Erd\H{o}s-Lov\'asz Tihany Conjecture for line graphs of multigraphs and has been proved by \cite{KS08}. So in the following, we just consider that $\ell>0$. Let $s,t\geq 3.5\ell+2$ when $\ell\ge 2$ and $s\ge 4$ when $\ell=1$.   Suppose that $G$ is a counterexample to the theorems with fewest vertices. Then $G$ is connected.

For a vertex $v\in V(G)$, let $d(v)$ be the degree and $N(v)$ be the set of neighbors of $v$. Let $\Delta(G)$ be the maximum degree of $G$.  Note that $d(v)\ge |N(v)|$.  Let $E(v)=\{e\in E(G): e=uv$ for some $u\in N(v)\}$ and $E(uv)=\{e\in E(G):$ the endpoints of $e$ are $u$ and $v\}$. We denote $|E(uv)|=m(uv)$ and $|E(v)|=m(v)=d(v)$. For $S_v\subseteq E(v)$, let $V_{S_v}=\{u\in N(v): E(uv)\cap S_v\not=\emptyset\}$.

A triangle in $G$ consists of three mutually adjacent vertices. The maximum number of edges between vertices in triangles in $G$ will be denoted by $\tau(G)$. Let $\omega'(G)=\max\{\tau(G),\Delta(G)\}$. Then $\omega'(G)=\omega(L(G))$.  Note that $|E(G)|\ge \chi'(G)=\chi(L(G))=s+t-1\ge 7\ell+1$. By Shannon's theorem \cite{Sh}, $\lfloor\frac{3\Delta(G)}2\rfloor\ge \chi'(G)=s+t-1\ge 2s-1$. So $s\leq\Delta(G)$.

For all vertices with degree $\Delta(G)$, we choose $v$ such that $|N(v)|$ is as large as possible. Let $N(v)=\{v_1, \ldots, v_d\}$. We also assume that $m(vv_1)\geq m(vv_2)\geq\ldots \geq m(vv_{d})$.

Next, we pick $S_v\subseteq E(v)$ with $s$ edges according to the following rules.  We pick one edge for $S_v$ successively from each edge set of $E(vv_1), \ldots, E(vv_{d})$ and delete the selected edges. We repeat the above step in the remaining graph. If we cannot pick edges in $E(vv_i)$ for some $i\in [d]$, then let us start a new cycle to pick edges from $E(vv_1)$ until we have selected $s$ edges. Let $S_v\cap E(vv_i)=S(vv_i), |S(vv_i)|=s(vv_i)$. Thus, we have that $s(vv_1)\geq s(vv_2)\geq\ldots \geq s(vv_{d})\ge 1$, and furthermore, we have the following useful fact:
\begin{center}
For each $1\le i<j\le d-1$, if $s(vv_i)-s(vv_{j})\ge 2$, then $s(vv_{j})=m(vv_{j})$.
\end{center}

We shall consider the edge-coloring of $G$, which is equivalent to the vertex-coloring of $L(G)$. Since $G$ is a counterexample and $S_v$ forms a clique in $L(G)$, we have $\chi'(G-S_v) \leq t+\ell-1$. Let $G' = G-S_v$, and let $\varphi$ : $E(G')\rightarrow\{1,\ldots ,t+\ell-1\}$ be a proper $(t+\ell-1)$-edge-colouring of $G'$. For vertex $x\in V(G)$, let
$$\varphi(x)=\{\varphi(e):  e\in E(G')\cap E(x)\} \text{ and } \overline{\varphi}(x) =\{1,\ldots, t+\ell-1\}- \varphi(x)$$

Since $s+t-1 = \chi'(G) > \omega'(G) \geq \Delta (G)\ge d(v)$ and all $s$ edges of $S_v$ are incident with $v$,
\begin{equation}\label{a}
  |\overline{\varphi}(v)|\geq t+\ell-1-(d(v)-s)\geq t+\ell-1-(s+t-1-s)=\ell+1.
\end{equation}

We denote $\overline{\varphi}(v)=\{c_1, \ldots, c_{\ell+1}, \ldots\}$. %Furthermore, let the edge set of $s$-star $S_v=S_0\cup T_0$ such that $|S_0|=\ell+1$ and $|T_0|=s-\ell-1$.
Let $\alpha_1, \alpha_2,\ldots ,\alpha_{s-\ell-1}$ be colors different from $[t+\ell-1]$, which together with $[t+\ell-1]$ gives $s+t-2$ colors. We will extend the edge-coloring $\varphi$ to a proper edge-coloring $\phi$ of $G$. We can always choose a set of $s-\ell-1$ edges $T_0\subseteq S_v$ and color them with $\{\alpha_1,\alpha_2,\ldots ,\alpha_{s-\ell-1}\}$. We will specify how to choose $T_0$, but once chosen, we let $T_0=\{e_{\alpha_1},e_{\alpha_2},\ldots ,e_{\alpha_{s-\ell-1}}\}$ such that $\phi(e_{\alpha_i})=\alpha_i$. Let $S_0=S_v-T_0=\{e_1,e_2,\ldots ,e_{\ell+1}\}$ be the set of remaining uncolored edges in $S_v$.

\begin{claim}\label{c1}
 $|N(v)|\leq \ell+2$.
\end{claim}

\begin{proof}
Suppose to the contrary that $|N(v)|=d\geq \ell+3$. By the choice of $S_v$,  $|V_{S_v}|=d'= \min\{d,s\} \geq \ell+3$.  Let $V_{S_v}=\{v_1,\ldots ,v_{d'}\}$. We show that $G$ has a proper $(s+t-2)$-edge-coloring.

Suppose that $\ell=1$ and $4\le s\le5$. Note that $|S_0|=\ell+1=2$ and $2\le|T_0|\le3$. Let $S_0=\{e_1,e_2\}$ and $e_i\in E(vv_i)$ for $i\in\{1,2\}$.
For $1\le i\le |T_0|$, we color one edge in $E(vv_{i+2})$ with $\alpha_i$ and denote the edge by $e_{\alpha_i}$. We try to color $e_i\in S_0$ with $c_i\in \overline{\varphi}(v)$ such $\varphi(e_i)=c_i$. It is not possible only when $c_i$ appears on an edge, say $e_i'$, incident with the endpoint of $e_i$. We construct a bipartite graph $T$ with parts $S_0=\{e_1,e_2\}$ and $T_0=\{e_{\alpha_1},e_{\alpha_2}\}$ such that $e_ie_{\alpha_j}\in E(T)$ if and only if there is no edge colored $c_i$ that is incident with the endpoints of $e_i$ and $e_{\alpha_j}$. Since $d_{T}(e_i)\geq 1$, $T$ contains a matching edge, say $e_1e_{\alpha_{1}}$. We also know that if there is a matching edge saturating $e_i$, then we can color $e_i$ with $c_i$ and recolor $e_i'$ with the color of the edge that is matched with $e_i$. Therefore, $e_2$ is not covered by a matching edge, that is, $e_2$ is not adjacent to $e_{\alpha_2}$ in $T$. And we also know that $e_1$ is not adjacent to $e_{\alpha_2}$ in $T$. 
So $e_1, e'_1, e_{\alpha_2}$ belong to the same triangle, say $vv_1v_4$. And $e_2,e'_2,e_{\alpha_2}$ belong to the same triangle, say $vv_2v_4$. 
Now we color $e_1$ with $c_1$ and recolor $e_1'$ with $\alpha_1$. And then consider the subgraph $H_{c_2,\alpha_1}$ induced by edges colored with $c_2$ and $\alpha_1$. The components are path and even cycles, and a component, say $P$, contains $v_2$. We interchange the colors $c_2$ and $\alpha_1$ on $P$ directly. So $e'_2$ is colored with $\alpha_1$ now. Thus, $e_2$ can be colored with $c_2$. Thus we obtain an $(s+t-2)$-coloring of $E(G)$, a contradiction. 

Otherwise, $s\ge 3.5\ell+2$. For $1\le i\le d'$, we color one edge in $E(vv_i)$ with $\alpha_i$ and denote the edge by $e_{\alpha_i}$. Among all the uncolored edges in $E(v)$, we choose edge set $S_0$ with $|S_0|=\ell+1$ such that $V_{S_0}=\{v_i: vv_i\in S_0\}$ is maximized.  Label the edges of $S_0$ as $e_1,e_2,\ldots ,e_{\ell+1}$. Then we color the uncolored edges of $S_v\backslash S_0$ with $\alpha_{d'+1},\ldots ,\alpha_{s-\ell-1}$ arbitrarily and denote by $e_{\alpha_i}, i\in \{d'+1,\ldots ,s-\ell-1\}$.

We try to color $e_i\in S_0$ with $c_i\in \overline{\varphi}(v)$. It is not possible only when $c_i$ appears on an edge, say $e_i'$, incident with the endpoint of $e_i$ that it is not $v$. %Otherwise, we have that either some edge $e_u$ between $v_iu$ colored be $c_i$ for some $u\notin V_{S_v}$ or some edge $e_{v_j}$ between $v_i$ and $v_j$ colored by $c_i$ for some $v_j\in V_{S_v}$. Now we try to recolor the edge $e_u$ or $e_{v_j}$ with a color in $\{\alpha_1, \alpha_2,\ldots ,\alpha_{s-\ell-1}\}$.
We construct a bipartite graph $T$ with parts $S_0=\{e_1,\ldots ,e_{\ell+1}\}$ and $T_0=\{e_{\alpha_1},\ldots ,e_{\alpha_{s-\ell-1}}\}$. Let $e_ie_{\alpha_j}\in E(T)$  if and only if there is no edge colored $c_i$ that is incident with the endpoints of $e_i$ and $e_{\alpha_j}$.
Since $d_{T}(e_i)\geq \min\{d',s-\ell-1\}-2\geq \ell+1=|S_0|$, $T$ contains a matching saturating $S_0$ by Hall's Theorem. Now, we can color $e_i$ with $c_i$ and recolor $e_i'$  with the color of the edge that is matched with $e_i$.  Therefore, $G$ has an $(s+t-2)$-edge-coloring, a contradiction.
\end{proof}

\begin{claim}\label{c2}
 $|N(v)|\ge 3$. %$s(vv_1)\leq \ell$ for $\ell\ge 2$.
 \end{claim}

\begin{proof}
Assume that $|N(v)|\le 2$.  Then $s(vv_1)\geq \frac{s}{2}\ge\ell+1$. Let $S_0\subseteq S(vv_1)$.   Let $|S(vv_2)\cap T_0|=\min\{s(vv_2),s-\ell-1\}$, and
\begin{center}
$A(v_1v_2)=\{\alpha:\alpha\notin \varphi(v)$ and $\alpha$ does not appear on the edges between $v_1$ and $v_2\}$.
\end{center}
Then
\begin{align*}
|A(v_1v_2)|&\ge (t+\ell-1)-(d(v)-s)-m(v_1v_2)=(t+s-1)+\ell-(d(v)+m(v_1v_2))\\
&\ge (t+s-1)+\ell-\tau(G)\ge(t+s-1)+\ell-(t+s-2)=\ell+1.
\end{align*}

Furthermore, as $d(v)\ge d(v_1)$,
\begin{equation}\label{a1}
  m(vv_2)=d(v)-m(vv_1)\ge d(v_1)-(m(vv_1)+m(v_1v_2)).
\end{equation}

If $s(vv_2)\geq \ell+1$, then we can make $\ell+1$ colors in $A(v_1v_2)$ available to use on uncolored edges in $E(vv_1)$, by recoloring the edges at $v_1$ with colors on edges in $S(vv_2)$. On the other hand, if $s(vv_2)<\ell+1$, then we have $$
s(vv_1)-s(vv_2)=s-2s(vv_2)\ge \begin{cases}1.5\ell+2\ge 2, \text{ if $\ell\ge 2$ },\\ 4-2=2, \text { if $\ell=1$}.
\end{cases}$$
It follows from the choice of $S_v$ that $m(vv_2)=s(vv_2)$.  So for colors appearing at edges in $E(v_1)-(E(v_1v_2)\cup E(vv_1))$, we may recolor them with distinct colors at $S(vv_2)$ by \eqref{a1}. Therefore, we can always make $\ell+1$ colors available for the $\ell+1$ uncolored edges in $E(vv_1)$, and obtain an $(s+t-2)$-coloring of $E(G)$, a contradiction.
\end{proof}

\begin{claim}\label{c23}
$\ell\ge 2$.
\end{claim}

\begin{proof}
Suppose that $\ell=1$. By Claims \ref{c1} and \ref{c2}, $|N(v)|=3$, and from \eqref{a}, $|\overline{\varphi}(v)|\geq 2$. Since $s\geq4$, we have that $s(vv_1)\geq 2$.  Choose $S_0\subseteq S(vv_1)$.   For $i=1,2$, we color one edge in $S(vv_{i+1})$ with $\alpha_{i}$ and denote $e_{\alpha_i}$ the edge. Then we color the uncolored edges of $S_v-S_0$ with $\alpha_{3},\ldots ,\alpha_{s-2}$ arbitrarily and denote $e_{\alpha_i}$ the edge of color $\alpha_i$.

We try to color $e_i\in S_0$ with $c_i\in \overline{\varphi}(v)$ such $\varphi(e_i)=c_i$. It is not possible only when $c_i$ appears on an edge, say $e_i'$, incident with the endpoint of $e_i$. We construct a bipartite graph $T$ with parts $S_0=\{e_1,e_2\}$ and $T_0=\{e_{\alpha_1},\ldots ,e_{\alpha_{s-2}}\}$ such that $e_ie_{\alpha_j}\in E(T)$ if and only if there is no edge colored $c_i$ that is incident with the endpoints of $e_i$ and $e_{\alpha_j}$. Since $d_{T}(e_i)\geq d-2=1$, $T$ contains a matching edge $e_1e_{\alpha_{j_0}}$. We also know that if there is a matching edge saturating $e_i$, then we can color $e_i$ with $c_i$ and recolor $e_i'$ with the color of the edge that is matched with $e_i$. Therefore, $e_2$ is not covered by a matching edge, that is, $e_2$ is not adjacent to $e_{\alpha_{j}}$ for every $j\neq j_0$ in $T$. So $e_2,e'_2$ and $T_0-e_{\alpha_{j_0}}$ belong to the same triangle, say $vv_1v_2$. It follows that $s(vv_3)=1$.

Let $A(v_1v_2)=\{\alpha:\alpha\notin \varphi(vv_1)\cup\varphi(vv_2)\cup \varphi(v_1v_2)\}$, and $r=|A(v_1v_2)|$. Then
\begin{align*}
r &\ge t-(m(vv_1)-s(vv_1))-(m(vv_2)-s(vv_2))-m(v_1v_2)\\
& =t+s(vv_1)+s(vv_2)-(m(vv_1)+m(vv_2)+m(v_1v_2))\ge t+s(vv_1)+s(vv_2)-\tau(G)\\
&\ge t+s(vv_1)+s(vv_2)-(s+t-2)=s(vv_1)+s(vv_2)-s+2=s-1-s+2=1.
\end{align*}

Let $\alpha\in A(v_1v_2)$. Consider the subgraph $H_{c_2,\alpha}$ induced by edges colored with $c_2$ and $\alpha$. The components are path and even cycles, and a component, say $P$, contains $v_1$. If $P$ contains $vv_{3}$, then $\varphi(vv_3)=\alpha$ and there is an edge $e^*$ incident with $v_{3}$ colored with $c_2$. Thus, we can use $\alpha_{1}$ to recolor $e^*$, and interchange the colors $c_2$ and $\alpha$ on $P$. If $P$ dose not contain $vv_3$, then we interchange the colors $c_2$ and $\alpha$ on $P$ directly. So the edge incident with $v_1$ that is colored by $c_2$ can be recolored with $\alpha_{3}$. Thus, $e_2$ can be colored with $c_2$. Thus we obtain an $(s+t-2)$-coloring of $E(G)$, a contradiction.
\end{proof}

Theorem~\ref{th2} has been proved following Claim~\ref{c23}.  We assume that $|N(v)|\ge 3$ and $\ell\ge 2$.

\begin{claim}\label{c3}
 $s(vv_1)\le \ell$.
\end{claim}

\begin{proof}
Assume that $s(vv_1)\geq \ell+1$. Let $S_0\subseteq S(vv_1)$.
For $1\le i\le d-1$, we pick one edge in $E(vv_{i+1})$ and color it with $\alpha_{i}$, and denote $e_{\alpha_i}$ this edge. Then we color the uncolored edges of $S_v-S_0$ with $\alpha_{d},\ldots ,\alpha_{s-\ell-1}$ arbitrarily and denote $e_{\alpha_i}$ the edge of color $\alpha_i$.

We try to color $S_0$ with $\{c_1,\ldots ,c_{\ell+1}\}\in \overline{\varphi}(v)$.  For each $i$, if $c_i$ is available at $v_1$, then we color  one edge in $S_0$ with $c_i$.  Assume that $k$ colors $c_{p+1},\ldots ,c_{\ell+1}$ have been used on edges $e_{p+i}, \ldots, e_{\ell+1}$ for $p=\ell+1-k$, respectively.  It follows that $d(v_1)\geq m(vv_1)+p$.

For $1\le i\le p$, we try to use $c_i$ to color $e_i$. As $e_i$ cannot be colored with $c_i$, the color $c_i$ appears on an edge, say $e_i'$, that is adjacent to $e_i$. % such that $\phi(e_i)=c_i$,
We construct a bipartite graph $T$ with part $S'_0=\{e_1,\ldots ,e_{p}\}$ and $T'_0=\{e_{\alpha_1},\ldots ,e_{\alpha_{s-\ell-1}}\}$ such that $e_ie_{\alpha_j}\in E(T)$ if and only if no edge colored $c_i$ is between the endpoints of $e_i$ and $e_{\alpha_j}$.

Let $(S''_0,T''_0)$ be the bipartite graph obtained  from $T$ by deleting the endpoints of a maximum matching of $(S'_0,T'_0)$. Let $S''_0=\{e_1,\ldots ,e_q\}\subseteq S'_0$ and $T''_0=\{e_{\alpha_1},\ldots ,e_{\alpha_{q'}}\}\subseteq T'_0$, where $q'=s-2\ell+q+k-2$. Then there are no matching edges between $S_0''$ and $T_0''$.  It follows that the edge colored $c_i$ is between $v_1$ and the endpoint of every edge in $S''_0$ for $i\in[q]$.  That is, there is a vertex $v_j$ for $j\in\{2,\ldots,d\}$ such that all colors of $\{c_1,\ldots ,c_q\}$ appear on the edges between $v_1$ and $v_j$.

Let $\ell'=\sum\limits_{i\neq1,j} s(vv_i)$. Then $\ell'= p-q$ due to the choice of $(S''_0,T''_0)$. It follows that $$p-q=\ell'=\sum_{i\neq1,j}s(vv_i)=(s-\ell-1)-(s(vv_1)-\ell-1)-s(vv_j).$$

As $c_1$ is on an edge between $v_1$ and $v_j$, the neighbors of $e_1$ in $T_0'$ are $\cup_{i\not=1,j}S(vv_i)$.  Since $e_1$ cannot be matched in $T_0'$, $e_1$ has at most $\ell$ neighbors in $T_0'$. Therefore
$$\ell\geq d_{T'_0}(e_1)=\sum_{i\neq1,j}s(vv_i)=(s-\ell-1)-(s(vv_1)-\ell-1)-s(vv_j)=\ell'=p-q.$$

As $s\ge 3.5\ell+2$, $s(vv_1)+s(vv_j)=s-\ell'\ge s-\ell\ge 2.5\ell+2$. Since $s(vv_1)\ge s(vv_j)$, we have $s(vv_1)\ge \ell+2\ge \ell'+2$.  From the way how $S_v$ is chosen, all the edges in $\cup_{i\not=1,j} E(vv_i)$ are in $S_v$.  So $s(vv_i)=m(vv_i)$ for each $i\not=1,j$.  That is,  $$\sum_{i\neq1,j}m(vv_i)=\ell'.$$ It follows that
$$\sum_{i=1}^{d}m(vv_i)=m(vv_1)+m(vv_j)+\ell'=d(v)=\Delta(G)\ge d(v_1)\ge m(vv_1)+p. $$
Then we have $m(vv_j)\geq p-\ell'= q$. Note that if $s(vv_1)\ge s(vv_j)+2$, then $s(vv_j)=m(vv_j)\ge q$ by the choice of $S_v$, and if $s(vv_1)\le s(vv_j)+1$, then $s(vv_j)\ge s(vv_1)-1\ge \ell+1=p+k\ge q$. Therefore, we always have  $s(vv_j)\geq q.$ We may assume that $\alpha_1', \ldots, \alpha_q'$ are colored appeared in some edges in $S(vv_j)$.

%For $j\in\{2,\ldots ,d'\}$,
Let $A(v_1v_j)$ be the set of colors that do not appear at edges in $E(v)\cup E(v_1v_j)$. Let $r=|A(v_1v_j)|$.
For each $\alpha\in A(v_1v_j)$ and $i\in [q]$, the subgraph $H_{c_i,\alpha}$ formed by edges colored with $c_i$ or $\alpha$ contains a component $P$ with $v_1\in P$. Clearly, $P$ is a path or an even cycle.  Interchange the colors $c_i$ and $\alpha$ on $P$, and then recolor the edge at $v_1$ of color $c_i$ with $\alpha_{i}$, for some $i\in[q']$, we can now color $e_i$ with $c_i$. Since $s(vv_j)\geq q$, we can follow this procedure to color every $e_i$ with $c_i$ for $i\in [q]$ as long as $r\ge q$. Thus, we may assume that $r<q$. From $r\ge |\overline{\varphi}(v)|-k-m(v_1v_j)$, we have $$m(v_1v_j)\geq |\overline{\varphi}(v)|-k-r\ge t+\ell-1-(d(v)-s)-k-r.$$ From $\ell+1=p+k=q+\ell'+k>r+\ell'+k$, we have
\begin{align*}
\tau(G)&= m(vv_1)+m(vv_j)+m(v_1v_j)\geq d(v)-\ell'+m(v_1v_j)\\
&\ge d(v)-\ell'+t+s-1+\ell-k-r-d(v)=t+s-1+\ell-\ell'-k-r>t+s-1,
\end{align*}
a contradiction.
\end{proof}

For each partition of $S_v$ into $S_0$ and $T_0$ with $|S_0|=\ell+1$, we construct a bipartite graph $T$ with part $S_0=\{e_1,\ldots ,e_{\ell+1}\}$ and $T_0=\{e_{\alpha_1},\ldots ,e_{\alpha_{s-\ell-1}}\}$ such that $e_ie_{\alpha_j}\in E(T)$ if and only if no edge colored with $c_i$ is incident with the endpoints of $e_i$ and $e_{\alpha_j}$. Thus, if we can find a matching in $T$ that saturates $S_0$, then we can color $e_i$ with $c_i$ and recolor the edge at endpoints of $e_i$ that has color $c_i$ with the color of the edge that is matched with $e_i$ in $T$, which would yield an $(s+t-2)$-coloring of $E(G)$, a contradiction. Therefore, we may assume that no matter how to select $S_0$ from $S_v$, there is no matching saturating $S_0$ in the bipartite graph $T$. From Hall's Theorem, $d_T(e_i)<\ell+1$. Since  for each $i\in [\ell+1]$,  $$d_T(e_i)\ge |S_v|-|S_0|-\max\{|T_0\cap S(vv_i)|+|T_0\cap S(vv_j)|: 1\le i<j\le d\},$$
and by Hall's Theorem, $d_T(e_i)<\ell+1$ for some $i\in [\ell+1]$,
we know that for each partition of $S_v$,
\begin{equation}\label{eq0}
\max\{|T_0\cap S(vv_i)|+|T_0\cap S(vv_j)|: 1\le i<j\le d\}\ge |S_v|-|S_0|-\ell=s-2\ell-1.
\end{equation}

We now show that we can always find a partition of $S_v$ to fail \eqref{eq0}.  Here is how we choose $S_0$:
If $s(vv_1)=\ell$, then let $S_0=S(vv_1)$. Otherwise, there is an integer $i_0>1$ such that
\begin{center}
$\ell\le\sum\limits_{i=1}^{i_0}s(vv_i)$ and $\sum\limits_{i=1}^{i_0-1}s(vv_i)\le\ell-1.$
\end{center}
By the choice of $S_v$, $\sum\limits_{i=1}^{i_0}s(vv_i)<2\ell.$ Thus, we can choose $S_0$ of size $\ell$ so that $|S_0\cap S(vv_i)|\ge\lfloor\frac{s(vv_i)}2\rfloor$ for $i\in[i_0]$ and $S_0\cap S(vv_i)=\emptyset$ for $i\in\{i_0+1,\ldots,d\}$.

%\begin{center}
%let $|S_0\cap S(vv_i)|=\frac12 s(vv_i)$ for $i=1,2$, and take up to $\frac12s(vv_i)$, if needed, for each $i\ge 3$.
%\end{center}

Since $s\ge 3.5\ell+2$ and $s(vv_1)\le \ell$ by Claim~\ref{c2}, $d=|N(v)|\ge 4$ and $|S_v\cap E(vv_i)|\ge 1$ for at least four vertices $v_i\in N(v)$. Clearly, for $1\le i<j\le d$,  $$|T_0\cap S(vv_i)|+|T_0\cap S(vv_j)|\le 0.5\ell+\ell=1.5\ell<1.5\ell+1\le s-2\ell-1,$$ a contradiction to \eqref{eq0}.

\bigskip

\section{Final Remarks}\label{sect2}

An enhanced version of the Erd\H{o}s-Lov\'asz Tihany Conjecture would be the following.  For each integer $\ell$, there exists an integer $f(\ell)$, such that for every graph $G$ with $\chi(G) > \omega(G)$ and any two integers $s,t\geq f(\ell)$ with $s+t =\chi(G)+1$, there is a partition $(S,T)$ of the vertex set $V(G)$ such that $\chi(G[S])\geq s$ and $\chi(G[T])\geq t+\ell$.  In Theorem~\ref{th1}, we obtained $f(\ell)\le 3.5\ell+2$ for line graphs of multigraphs when $\ell\ge 0$. One immediate question is to determine $f(\ell)$ for line graphs.  It is also interesting to know if $f(\ell)$ exists for other classes of graphs. %One may also consider the case when $\ell<0$.

%We have obtained a linear lower bound for the line graphs of multigraphs, but we do not know whether $4(\ell+1)$ is the optimal bound for $s,t$. Maybe we can find the better lower bound. Furthermore, what most interests us is whether this is still true for other graph classes.

\end{document}